\documentclass[10pt]{article}
\usepackage{mathrsfs}

\usepackage{latexsym}
\usepackage{bbm}
\usepackage{amsmath}
\usepackage{amsfonts}
\usepackage{amssymb}
\usepackage{amsthm}
\usepackage{multirow}
\usepackage{authblk}
\usepackage{latexsym}
\usepackage[dvips]{graphicx}
\usepackage{overpic}
\usepackage{graphicx}
\usepackage{hyperref}
\usepackage{geometry}
\newtheorem{theorem}{\bf Theorem}[section]
\newtheorem{lem}[theorem]{\bf Lemma}

\newtheorem{coro}[theorem]{\bf Corollary}

\newtheorem{rem}[theorem]{\bf Remark}

\newcommand\diff{\,\mathrm{d}}

\def \bN {\Bbb N}
\def \bR {\Bbb R}
\def \bN {\Bbb N}

\def \and {\, \mbox{\rm and}\, }

\def \l {\left}
\def \r {\right}
\def \bR {\Bbb R}
\def \bN {\Bbb N}

\def \l {\left}
\def \r {\right}

\makeatletter

\newcommand{\Rmnum}[1]{\expandafter\@slowromancap\romannumeral #1@}
\makeatother

\begin{document}

\title{\bf The sharp constant for  truncated Hardy-Littlewood maximal inequality}
\author[a]{Jia Wu}
\author[a]{Shao Liu\thanks{Corresponding author, E-mail: {liushao19@mails.ucas.ac.cn}}}
\author[b]{Mingquan Wei}
\author[a]{Dunyan Yan}
\affil[a]{\it \small School of Mathematical Sciences, University of Chinese Academy of Sciences, Beijing, {\rm 100049}, China}
\affil[b]{\it \small School of Mathematics and stastics, Xinyang Normal University, Xinyang, {\rm 464000}, China}
\date{}

\maketitle

\vspace{-1em}
\begin{abstract}
\noindent This paper focuses on the operator norm of the truncated  Hardy-Littlewood maximal operator $M^b_a$ and the strong truncated Hardy-Littlewood maximal operator $\tilde{M}^{\boldsymbol{b}}_{\boldsymbol{a}}$, respectively. We first present the $L^1$-norm of $M^b_a$, and then the $L^1$-norm of $\tilde{M}^{\boldsymbol{b}}_{\boldsymbol{a}}$ is given. Our study may have some enlightening significance for the research on sharp constant for the classical Hardy-Littlewood maximal inequality.

\vspace{0.3em}

\vspace{0.3em}
\end{abstract}

\vspace*{0.5mm}

\section{Introduction}
\label{sec1}

The purpose of this paper is to establish the $L^1$-norm, respectively, of  the truncated Hardy-Littlewood maximal operator and the strong truncated Hardy-Littlewood
maximal operator.

We begin with definitions of the maximal operators studied here. Given two parameters $a$ and $b$ satisfying $0<a\leq b<\infty$, for a locally integrable function $f$ on $\bR^n$,  the truncated Hardy-Littlewood maximal operator $M_a^b$ is defined by
\begin{equation}\label{001}
M_a^bf(x)=\sup_{a<r<b}\frac{1}{|B(x,r)|}\int_{B(x,r)}|f(y)|\diff{y},
\end{equation}
where $B(x,r)$ denotes the ball centered at $x\in \bR^n$ with the radius $r$.

Suppose that $\boldsymbol{a}=(a_1,a_2,\dots,a_n)$ and $\boldsymbol{b}=(b_1,b_2,\dots,b_n)$ are two vectors in $\mathbb{R}^n$. We call ${\bf 0}<\boldsymbol{a}\leq\boldsymbol{b}<\boldsymbol{\infty}$ if $0<a_i\leq b_i<\infty$ for each $i=1,2,\dots,n$. For ${\bf 0}<\boldsymbol{a}\leq \boldsymbol{b}<\boldsymbol{\infty}$, we further give the definition of  the strong truncated Hardy-Littlewood  maximal operator ${M}^{\boldsymbol{b}}_{\boldsymbol{a}}$, that is,
\begin{equation}\label{001*}
{M}^{\boldsymbol{b}}_{\boldsymbol{a}}f(x)=\sup_{R\in \mathscr{R}(x)}\frac {1}{|R|}\int_{R}|f(y)|\diff{y},
\end{equation}
where $\mathscr{R}(x)$ is the set of all rectangles $R$ in $\bR^n$ which satisfies the following three properties:
\vspace{-0.3em}
\begin{enumerate}
\setlength{\itemsep}{1.5pt}
\setlength{\parsep}{0pt}
\setlength{\parskip}{0pt}
\setlength{\topsep}{1ex}
\item[(\Rmnum{1})] $R$ is centered at $x\in\mathbb{R}^n$;
\item[(\Rmnum{2})] All sides of $R$ are parallel to axes;
\item[(\Rmnum{3})] Denote by $\ell_i$ ($1\leq i\leq n$) the length of each side of $R$, then $a_i<\frac {\ell_i}{2}<b_i$.
\end{enumerate}

\vspace{-0.3em}
Let $\bR^n=\bR^{n_1}\times \bR^{n_2}\times\cdots\times  \bR^{n_k}$, where  $1\leq n_i\leq n$ for $1\leq i\leq k$. Suppose that $\boldsymbol{a}$ and $\boldsymbol{b}$ are two vectors in $\bR^k$ such that ${\bf 0}<\boldsymbol{a}\leq\boldsymbol{b}<\boldsymbol{\infty}$. For any $x\in\bR^n$, we have $x=(x_1,\dots,x_k)$ where $x_i\in\bR^{n_i}$ for $1\leq i\leq k$. Now we are in the position to give the definition of the general strong truncated Hardy-Littlewood  maximal operator $\tilde{M}^{\boldsymbol{b}}_{\boldsymbol{a}}$, namely,
\begin{equation}
	\tilde{M}^{\boldsymbol{b}}_{\boldsymbol{a}}f(x)=\sup_{a_i\leq r_i\leq b_i\atop 1\leq i\leq k}\frac{1}{\prod_{i=1}^{k}|B(x_i,r_i)|}\int_{\prod_{i=1}^kB(x_i,r_i)}|f(y)|\diff{y},
\end{equation}
where $\prod_{i=1}^kB(x_i,r_i)=B(x_1,r_1)\times \cdots \times B(x_k,r_k)$ and $B(x_i,r_i)$ denotes the ball centered at $x_i\in \bR^{n_i}$ with the radius $r_i$.

Clearly, if each $n_i=1$, then $\tilde{M}^{\boldsymbol{b}}_{\boldsymbol{a}}=M_a^b$. If $k=1$, then $\tilde{M}^{\boldsymbol{b}}_{\boldsymbol{a}}=M^{\boldsymbol{b}}_{\boldsymbol{a}}$. Hence, $\tilde{M}^{\boldsymbol{b}}_{\boldsymbol{a}}$ is exactly an extension of both operators $M_a^b$ and ${M}^{\boldsymbol{b}}_{\boldsymbol{a}}$.

Recall that for a locally integrable function $f$, the classical Hardy-Littlewood maximal operator $M$ and the classical strong Hardy-Littlewood  maximal
operator $\hat{M}$ are defined, respectively, by
\begin{equation}
\label{def000003}
Mf(x)=\sup_{r>0}\frac{1}{|B(x,r)|}\int_{B(x,r)}|f(y)|\diff{y},
\end{equation}
and
\begin{equation}
\tilde{M}f(x)=\sup_{R(x)}\frac{1}{|R(x)|}\int_{R(x)}|f(y)|\diff{y},
\end{equation}
where the supremum is taken over all rectangles $R(x)$ centered at $x\in\mathbb{R}^n$ with sides parallel to axes.

The classical (strong) Hardy-Littlewood  maximal operators are fundamental tools to study harmonic analysis, potential theory, and the theory of partial differential equations (see \cite{Adams1996,loukas2014classical}). Meanwhile, how to compute the operator norm of the classical Hardy-Littlewood maximal operator is a basic problem. There are fruitful results on the $L^p(\bR^n)\to L^p(\bR^n)$-norm of $M$ ($1<p<\infty$), and the sharp weak $(1,1)$ constant for $M$; one may see \cite{aldaz1998remarks,2009The,2013On,grafakos1997best,melas2002centered,melas2003best,2016Best} and the references there for more results and details.

It is well-known that the truncated operators have many crucial properties. For instance, as shown in \cite{Yan}, $L^p$-boundedness of the truncated operator and the corresponding oscillatory operator is equivalent. Hence, in order to facilitate and simplify the research on the classical maximal operators, one may study the truncated maximal operators and find their relations with the classical ones.

 Wei et al. \cite{wei2016note} and Zhang et al. \cite{zhang2020equivalence} first introduce the operator $M_0^{\alpha}$ ($\alpha >0$) which is defined by 
 \begin{gather}
 \label{def00002}
 M_0^{\alpha}f(x)=\sup_{0<r<\alpha}\frac{1}{|B(x,r)|}\int_{B(x,r)}|f(y)|\diff{y}.
 \end{gather}
 Obviously, $M_0^{\alpha}$ and $M$ can both be viewed as {\bf limit} or endpoint cases of $M_{a}^{b}$. 
As long as $0<a'\leq a$ and $ b\le b'\le\alpha<\infty$, in the pointwise sense, it can be deduced immediately from the definition \eqref{001}, \eqref{def000003} and \eqref{def00002} that
\begin{align*}
M_a^bf(x)&\leq M_{a'}^{b'}f(x)\leq M_0^{\alpha}f(x)\leq Mf(x)
\end{align*}
holds for any $x\in\bR^n$. Therefore it follows that
\begin{equation}
\label{001*}
\|M_a^b\|_{L^p(\bR^n)\to L^p(\bR^n)}\leq \|M_{a'}^{b'}\|_{L^p(\bR^n)\to L^p(\bR^n)}\leq \|M_{0}^{\alpha}\|_{L^p(\bR^n)\to L^p(\bR^n)}\leq\|M\|_{L^p(\bR^n)\to L^p(\bR^n)}.
\end{equation}
The same is true if we substitue the $L^p(\bR^n)\to L^p(\bR^n)$-norm in \eqref{001*} with the weak $(1,1)$-norm.  

It is noteworthy that Wei et al. \cite{wei2016note} shows the last inequality in \eqref{001*} is, in fact, an identity, that is, for $1<p<\infty$,
\begin{gather}
\label{wei001}
\l\|M_0^\alpha\r\|_{L^p(\bR^n)\to L^p(\bR^n)}=\l\|M\r\|_{L^p(\bR^n)\to L^p(\bR^n)},
\end{gather}
and 
\begin{gather}
\label{wei002}
\l\|M_0^\alpha\r\|_{L^1(\bR^n)\to L^{1,\infty}(\bR^n)}=\l\|M\r\|_{L^1(\bR^n)\to L^{1,\infty}(\bR^n)}.
\end{gather}
Identities \eqref{wei001} and \eqref{wei002} actually demonstrate that there is no difference between operators $M_0^{\alpha}$ and $M$ in the $L^p(\bR^n)\to L^p(\bR^n)$-norms  and weak $(1,1)$-norms, and this is essentially because $\frac {\alpha}{0}=\infty=\frac {\infty}{0}$.
 Motivated by these works, in this papar, we divert out attention to the operator $M_a^b$ for $0<a\le b<\infty$, and its general form $\tilde{M}^{\boldsymbol{b}}_{\boldsymbol{a}}$ for ${\bf 0}<\boldsymbol{a}\leq \boldsymbol{b}<\boldsymbol{\infty}$. Note that the essence of our discussion is using  the finity to deduce the infinity, which is the base of the calculus.

Now we formulate our main theorems as follows.
\begin{theorem}\label{main-2}
		For $0<a\le b<\infty$, we have that
\begin{equation}\label{G-4-1}
	\l\|M_a^b\r\|_{L^1(\bR^n)\to L^1(\bR^n)}=1+n\ln\frac{b}{a}.
\end{equation}
\end{theorem}

\begin{theorem}\label{main-3}
For ${\bf 0}<\boldsymbol{a}\leq \boldsymbol{b}<\boldsymbol{\infty}$, we have that
	\begin{equation}\label{G-4-1*}
		\l\|M_{\boldsymbol{a}}^{\boldsymbol{b}}\r\|_{L^1(\bR^n)\to L^1(\bR^n)}=\prod\limits _{i=1}^n\left(1+\ln\frac {b_i}{a_i}\right).
	\end{equation}
\end{theorem}

\begin{theorem}\label{main-4}
Let $n=n_1+n_2 \cdots+n_k$ for $ n_i\ge 1$ with $1\leq i\leq k$. Suppose that $\boldsymbol{a}$ and $\boldsymbol{b}$ are two vectors in $\bR^k$ such that $\bf{0}<\boldsymbol{a}\leq\boldsymbol{b}<\boldsymbol{\infty}$.
We have that
	\begin{equation}\label{G-4-1**}
		\l\|\tilde{M}_{\boldsymbol{a}}^{\boldsymbol{b}}\r\|_{L^1(\bR^n)\to L^1(\bR^n)}=\prod\limits _{i=1}^k\left(1+n_i\ln\frac {b_i}{a_i}\right).
	\end{equation}
\end{theorem}

\begin{rem}
Theorem \ref{main-4} covers both Theorem \ref{main-2} and Theorem \ref{main-3}.
\end{rem}
\begin{rem}
It is well known that the classical Hardy-Littlewood maximal operator $M$ is bounded from $L^1(\bR^n)$ to $L^{1,\infty}(\bR^n)$, but not bounded on $L^1(\bR^n)$. However, Theorem \ref{main-2} implies that the truncated Hardy-Littlewood maximal operator $M^b_a$ is bounded on $L^1(\bR^n)$. This clearly indicates the difference between $M$ and $M^b_a$. Indeed, previous literature mainly shows the similarity between the classical case and the truncated case, and our result demonstrates the truncated maximal operator has its own characteristics.
\end{rem}
Based on these theorems, one can directly obtain the following corollary.	
\begin{coro}\label{coro001}
	Set $\theta:=\frac ba$. Given $0<a\leq b<\infty$, we can show that
	\begin{equation*}
		\lim_{\theta\to \infty} \|M_a^b\|_{L^1(\bR^n)\to L^1(\bR^n)} =\infty,
	\end{equation*}
	and 
	\begin{gather*}
		\lim_{\theta\to \infty} \frac {\|M_a^b\|_{L^1(\bR^n)\to L^1(\bR^n)}}{\ln\frac ba} =n.
	\end{gather*}
\end{coro}
\begin{rem}
Corollary \ref{coro001} may be viewed as an intutive explanation of the unboundedness of $M$ and $M_0^{\alpha}$ on $L^1(\bR^n)$,  and it also shows the increasing rate of $\|M_a^b\|_{L^1}$ with respect to $\frac ba$ is $\log$-type.
\end{rem}

The rest of the present paper will be organized as follows: In section \ref{sec2}, we will give some important lemmas.
The $L^1(\bR^n)$ norm of $M_a^b$ will be calculated in section \ref{sec3}. Section \ref{sec4} shows the proof of $L^1(\bR^n)$ norm of $\tilde M_{\boldsymbol{a}}^{\boldsymbol{b}}$ which also contains the calculation of the $L^1(\bR^n)$ norm of $ M_{\boldsymbol{a}}^{\boldsymbol{b}}$.
A tacit understanding in the present paper is that all rectangles in $\bR^n$ are rectangles with sides parallel to the coordinate axes.
We denote by $|E|$ the Lebesgue measure of the measurable set $E\subset \bR^n$.

\section{Some facts and lemmas}\label{sec2}
Before we proceed to prove our main results,  several basic lemmas are needed. For simplicity, we assume that the function $f$ is non-negative in the proof of the following lemmas. Now, let $M_a^b$, $M_{\boldsymbol{a}}^{\boldsymbol{b}}$ and  $\tilde M_{\boldsymbol{a}}^{\boldsymbol{b}}$ as defined in Section \ref{sec1}.
\begin{lem}\label{L2.1}
	For $0<a\le b<\infty$ and ${\bf 0}<\boldsymbol{a}\leq \boldsymbol{b}<\boldsymbol{\infty}$,  we have that
\begin{align*}
\l\|M_a^b\r\|_{L^1}&\le \left(\frac ba\right)^n,\\
\l\|M_{\boldsymbol{a}}^{\boldsymbol{b}}\r\|_{L^1}&\leq \prod\limits _{i=1}^n\frac {b_i}{a_i},
\end{align*}
and
\begin{gather*}
\l\|\tilde M_{\boldsymbol{a}}^{\boldsymbol{b}}\r\|_{L^1}\leq \prod\limits _{i=1}^k\left(\frac {b_i}{a_i}\right)^{n_i}.
\end{gather*}
\end{lem}

\begin{proof}
		Let $f\in L^1(\bR^n)$. A direct computation implies that
 \begin{eqnarray*}
		\l\|M_a^bf\r\|_1&=&\int_{\bR^n}M_a^bf(x)\diff{x}\\
&=&\int_{\bR^n} \sup_{\atop{a<r<b}}\frac
		{1}{|B(x,r)|}\int_{B(x,r)}|f(y)|\diff{y}\diff{x}\\
		&\le& \int_{\bR^n} \frac {1}{|B(x,a)|}\int_{B(x,b)}|f(y)|\diff{y}\diff{x}\\
&=& \frac {1}{V_na^n}\int_{\bR^n} \int_{B(x,b)}|f(y)|\diff{y}\diff{x}\\
		&=& \frac {1}{V_na^n}\int_{\bR^n} \int_{\bR^n}\chi_{B(x,b)}|f(y)|\diff{y}\diff{x}\\
		&=&\frac {1}{V_na^n}\int_{\bR^n}|f(y)|\diff{y} \int_{\bR^n}\chi_{B(x,b)}\diff{x}\\
		&=&\l(\frac{b}{a}\r)^n\|f\|_{1},
	\end{eqnarray*}
	where $V_n$ is the volume of the unit ball in $\bR^n$.
	
By the definition of  norm of operator, we have
\begin{gather*}
\l\|M_a^b\r\|_{L^1}={\sup_{\|f\|_1=1}}\l\|M_a^bf\r\|_1.
\end{gather*}
It follows that
\begin{gather*}
\l\|M_a^b\r\|_{L^1}\le \l(\frac{b}{a}\r)^n.
\end{gather*}
Using the similar method, it can be deduced
 \begin{gather*}
 \l\|M_{\boldsymbol{a}}^{\boldsymbol{b}}\r\|_{L^1}\leq \prod_{i=1}^n\frac {b_i}{a_i},
 \end{gather*}
  and
 \begin{gather*}
\l \|\tilde M_{\boldsymbol{a}}^{\boldsymbol{b}}\r\|_{L^1}\leq \prod _{i=1}^k\l(\frac {b_i}{a_i}\r)^{n_i}.
 \end{gather*}
\end{proof}

Let $C_c(\bR^n)$ be the collection of all continuous functions with compact support on $\bR^n$. It is clear  that $C_c(\bR^n)$ is dense in $L^p(\bR^n)$. From this simple observation, we can reduce our discussion into $C_c(\bR^n)$.

\begin{lem}\label{L2.2}
 For  ${\bf 0}<\boldsymbol{a}\leq \boldsymbol{b}<\boldsymbol{\infty}$ and $1\leq p<\infty$, we have
\begin{gather*}
\l\|\tilde M_{\boldsymbol{a}}^{\boldsymbol{b}}\r\|_{L^p}={\sup_{\varphi\in C_c(\bR^n)\colon\|\varphi\|_{p}=1}}\l\|\tilde M_{\boldsymbol{a}}^{\boldsymbol{b}}\varphi\r\|_p.
\end{gather*}

\end{lem}
The proof is just a repetition of the proof of Lemma 2.5 in \cite{wei2016note}, so we omit it.

The following lemma will be used to give the upper bound of $\|M_a^b\|_{L^1}$.
\begin{lem}\label{L2.3}
Suppose that $\{O_j\}_{j=1}^{N}$ is a sequence of $N$ disjoint rectangles in $\bR^n$ such that each $O_j$ is a translation of $O_1$. For $\alpha_j\ge0$ with $j=1,2,\dots,N$, set
\begin{gather*}
f=\sum \limits_{j=1}^N\alpha_j\chi_{O_j},
\end{gather*}
and
\begin{gather*}
g=\left(\sum_{j=1}^N\alpha_i\right)\chi_{O_1}.
\end{gather*}
Then for $0<a\leq b<\infty$, the following statements hold.
\begin{enumerate}
\item[\rm{(\Rmnum{1})}] $\|f\|_{1}=\|g\|_{1}$,
\item[\rm{(\Rmnum{2})}] $\|M_a^bf\|_{1}\le\|M_a^bg\|_{1}$.
\end{enumerate}
\end{lem}
\begin{proof}
We first give the proof of statement (\Rmnum{1}). A direct computation implies that
\begin{equation}\label{G-f}
	\|f\|_{1}=\int_{\bR^n}|f(x)|\diff{x}=\int_{\bR^n}\sum_{j=1}^N\alpha_j\chi_{O_j}(x)\diff{x}
	=\sum_{j=1}^N\int_{O_j}\alpha_j\diff{x}=\sum_{j=1}^N\alpha_j |O_j|.
\end{equation}
  Since each $O_j$ is a translation of $O_1$ and translation has the measure-invariance, we have $|O_j|=|O_1|$ for all $j=1,2,\dots,N$.
 Consequently, we obtain that
   \begin{equation}\label{G-g}
     \sum_{j=1}^N\alpha_j|O_j|=\sum_{j=1}^N\alpha_j|O_1|
     =\sum_{j=1}^N\int_{O_1}\alpha_i\diff{y}
     =\int_{\bR^n}\left(\sum_{j=1}^N\alpha_j\right)\chi_{O_1}\diff{y}
     =\|g\|_{1}.
  \end{equation}
Combining (\ref{G-f}) with (\ref{G-g}) yields $\|f\|_{1}=\|g\|_{1}$.

Next we prove statement (\Rmnum{2}). Using the sublinearity of $M_a^b$ and the Minkowski's inequality, we can obtain that
	 \begin{eqnarray*}
		\l\|M_a^bf\r\|_{1}&=&\l\|M_a^b\l(\sum_{j=1}^N\alpha_j\chi_{O_j}\r)\r\|_{1}\\
		&\le&\l\|\sum_{j=1}^N\alpha_j M_a^b(\chi_{O_j})\r\|_{1}\\
		&\le&\sum_{j=1}^N\alpha_j\l\| M_a^b(\chi_{O_j})\r\|_{1}\\
		&=&\left(\sum_{j=1}^N\alpha_j\right)\l\| M_a^b(\chi_{O_1})\r\|_{1}\\
		&=&\l\|\left(\sum_{j=1}^N\alpha_j\right) M_a^b(\chi_{O_1})\r\|_{1}\\
		&=&\l\| M_a^b\l(\left(\sum_{j=1}^N\alpha_j\right)\chi_{O_1}\r)\r\|_{1}\\
		&=&\l\|M_a^bg\r\|_{1},
	\end{eqnarray*}
which is exactly what we want to prove.
\end{proof}

For $ M_{\boldsymbol{a}}^{\boldsymbol{b}}$ and $\tilde M_{\boldsymbol{a}}^{\boldsymbol{b}}$, we have the same conclusion, and one can substitute $M_a^b$ in Lemma \ref{L2.3} by $ M_{\boldsymbol{a}}^{\boldsymbol{b}}$ or $\tilde  M_{\boldsymbol{a}}^{\boldsymbol{b}}$ to formulate the corresponding result.  Here we just omit their statements and proofs.
\section{The $L^1$-norm of the truncated maximal operator $M^b_a$}\label{sec3}
From Lemma \ref{L2.1}, we have  known that $M^b_a$ is bounded on $L^1(\bR^n)$. In this section, we will precisely calculate the $L^1$-norm of $M^b_a$.

\begin{proof}[Proof of Theorem \ref{main-2}]
	We first prove $\|M_a^bf\|_{1}$ is less than the right hand side of  (\ref{G-4-1}) for all $f\in L^1(\bR^n)$ satisfying $\|f\|_{1}=1$. Applying Lemma \ref{L2.2}, we can restrict our discussion on $C_c(\bR^n)$.
	
	 Let $\varphi\in C_c(\bR^n)$ satisfy $\|\varphi\|_{1}=1$. For all $\varepsilon>0$, there exists a simple function
	 \begin{gather*}
	 h'=\sum \limits_{i=1}^N\alpha_i\chi_{O_i}
	 \end{gather*}
	 such that $|\varphi|\leq h'$ and $\|h'\|_{1}\leq 1+\varepsilon$, where $\{\alpha_i\}_{i=1}^{N}$ and $\{O_i\}_{i=1}^{N}$ satisfy the assumptions in Lemma \ref{L2.3}. For the clarity of proof, we may further assume each $O_i$ is a cube.

  Obviously, there holds $\|M_a^b\varphi\|_{1}\le \|M_a^bh'\|_{1}$. For the function $h'$, using Lemma \ref{L2.3}, we may consider the function
  \begin{gather*}
  h=\left(\sum \limits_{i=1}^N\alpha_i\right)\chi_{O_1}
  \end{gather*}
   instead since $\|h\|_{1}=\|h'\|_{1}$ and $\|M_a^bh'\|_{1}\leq\|M_a^bh\|_{1}$.

  Now we further divide $O_1$ into $2^{nm}$ sub-cubes $O_{1l},~l=1,2,\dots,2^{nm}$, with the same side length. By choosing $m$ sufficiently large,
 the side length of each sub-cube $O_{1l}$ can be less than $\frac {a}{2\sqrt{n}}$. Take
 \begin{gather*}
 g=\left(2^{nm}\sum \limits_{i=1}^N\alpha_i\right)\chi_{O_{11}},
 \end{gather*}
using Lemma \ref{L2.3} again, we have $\|g\|_{1}=\|h\|_{1}$ and $\|M_a^bh\|_{1}\leq\|M_a^bg\|_{1}$.

 Without loss of generality, we assume that the cube $O_{11}$ is centered at the origin, since the $L^1$-norms of $g$ and $M_a^bg$ are invariant under the translation.
  Let the side length of $O_{11}$ be $2s$ and $s$ can be chosed sufficiently small.

  For the function $g$, we  can estimate $M_a^bg(x)$ via the scales of $x$. More specifically,  if $x\in B(0,a+\sqrt{n}s)$, then a simple computation implies that
 \begin{gather}
 \label{Ge-new-1}
 M_a^bg(x)\leq\frac{\|h\|_{1}}{{V_na^n}};
 \end{gather}
if  $x\in B(0,b+\sqrt{n}s)\backslash B(0,a+\sqrt{n}s)$, then it follows from a subtle calculation that
\begin{gather}
\label{Ge-new-2}
M_a^bg(x)\le \frac {\|h\|_{1}}{V_n({|x|-\sqrt{n}s)}^n};
\end{gather}
when $x\in B^c(0,b+\sqrt{n}s)$, we obviously have 
\begin{gather}
\label{Ge-new-3}
M_a^bg(x)=0.
\end{gather}
Based on estimates \eqref{Ge-new-1}, \eqref{Ge-new-2} and \eqref{Ge-new-3}, we choose the control function
  \begin{equation}
    u(x)=\left\{\begin{aligned}
  	 &\frac {\|h\|_{1}}{V_na^n}\quad & &\textrm{if }x\in B(0,a+\sqrt{n}s),\\
   	 &\frac {\|h\|_{1}}{V_n({|x|-\sqrt{n}s)}^n}\quad & &\textrm{if }x\in B(0,b+\sqrt{n}s)\backslash B(0,a+\sqrt{n}s),\\
  	 &0 \quad& &\textrm{if }x\in B^c(0,b+\sqrt{n}s).
   \end{aligned}
   \right.
  \end{equation}
Consequently, it immediately follows  that 
 \begin{gather*}
 \l\|M_a^bg\r\|_{1}\leq \|u\|_{1}.
 \end{gather*} 
Observe that the $L^1$-norm of $u$ can be estimated as follows:
\begin{equation}\label{G-4-fu}
	\|u\|_{1}=\left\{\begin{aligned}
		&\|h\|_{1}\left(\frac {(a+\sqrt{n}s)^n}{a^n}+n\left(\ln\frac {b}{a}+A(s)\right)\right) \quad &\textrm{if }n\geq 2,\\
		&\|h\|_{1}\l(\frac {a+s}{a}+\ln{\frac {b}{a}}\r) \quad &\textrm{if }n=1,
	\end{aligned}
	\right.
\end{equation}
where 
\begin{gather*}
A(s)=\sum \limits_{i=0}^{n-2}\left(C_{n-1}^{i} \frac {(\sqrt{n}s)^{n-i-1}}{(i+1-n)b^{n-i-1}}-C_{n-1}^{i} \frac {(\sqrt{n}s)^{n-i-1}}{(i+1-n)a^{n-i-1}}\right).
\end{gather*}
In fact, to obtain the identity \eqref{G-4-fu}, we conclude  that 
\begin{eqnarray*}
\int_{B(0, b+\sqrt{n}s)\backslash B(0,a+\sqrt{n}s) } \frac {1}{V_n({|x|-\sqrt{n}s)}^n}\diff{x}
&=&\frac {\omega_{n-1}}{V_n}\int_{a+\sqrt{n}s}^{b+\sqrt{n}s}\frac {r^{n-1}}{(r-\sqrt{n}s)^n}\diff{r}\\
&=&n\int_{a}^{b}\frac {(r+\sqrt{n}s)^{n-1}}{r^n}\diff{r}\\
&=&n\int_{a}^{b} \sum_{i=0}^{n-1}r^{i-n}(\sqrt{n}s)^{n-1-i}\diff{r}\\
&=&n\left(\ln\frac {b}{a}+A(s)\right),
\end{eqnarray*}
where $\omega_{n-1}$ is the surface area of the unit ball in $\bR^n$. One can also use sphere coordinates and Gamma function to deduce the same result.

   As a result, for each $\varphi\in C_c(\bR^n)$ with $\|\varphi\|_{1}=1$, we have   from the choice of $h$, $g$ and $u$ that
   \begin{equation}\label{G-4-les}
   \l\|M_a^b\varphi\r\|_{1}\le \l\|M_a^bh\r\|_{1}\le \l\|M_a^bg\r\|_{1}\le \|u\|_{1}.
   \end{equation}
  Thus, we can show from \eqref{G-4-fu}  and \eqref{G-4-les} that
   \begin{gather}
   \label{newaddf1}
   \l\|M_a^b\varphi\r\|_{1}\le \left\{\begin{aligned}
		&(1+\varepsilon)\left(\frac {(a+\sqrt{n}s)^n}{a^n}+n\left(\ln\frac {b}{a}+A(s)\right)\right) \quad &\textrm{if }n\geq 2,\\
		&(1+\varepsilon)\l(\frac {a+s}{a}+\ln{\frac {b}{a}}\r) \quad &\textrm{if }n=1.
	\end{aligned}
	\right.
   \end{gather}

   Let $s\rightarrow 0$ in (\ref{newaddf1}). We can deduce that
   	\begin{equation*}
   		\l\|M_a^b\varphi\r\|_{1}\le (1+\varepsilon)\l(1+n\ln\frac{b}{a}\r).
   \end{equation*}
   Since $\varepsilon>0$ is arbitrary, we immediately have that
   \begin{equation}\label{G-dan}
   	\l\|M_a^b\varphi\r\|_{1}\le\l(1+n\ln\frac{b}{a}\r).
   \end{equation}
Thus it naturally implies from (\ref{G-dan}) that
\begin{equation}\label{G-dan-1}
   	\l\|M_a^b\r\|_{L^1}\le \l(1+n\ln\frac{b}{a}\r).
   	\end{equation}

   Conversely, inspired by the shape and properties of function $u$, we construct a sequence of functions to deduce the inverse inequality. Let
   \begin{gather*}
   f_m=\frac {\chi_{B(0,1/m)}}{V_nm^{-n}},
   \end{gather*}
   for  $m=1,2,3,\dots$. Obviously,  $\|f_m\|_{1}=1$. For a sufficient large $m$, define
   \begin{equation*}
   	w(x)=\left\{\begin{aligned}
   		&\frac {1}{V_na^n} \quad &&\textrm{if }x\in B(0,a-1/m),\\
   		&\frac {1}{V_n({|x|+1/m})^n} \quad&&\textrm{if }x\in B(0,b-1/m)\backslash B(0,a-1/m),\\
   		&0 \quad&&\textrm{if }x\in B^c(0,b-1/m).
   	\end{aligned}
   	\right.
   \end{equation*}
  A direct computation yields that
 \begin{equation}\label{G-4-Mf}
  \l\|M_a^b\r\|_{L^1}\geq\l\|M_a^bf_m\r\|_{1}\ge \|w\|_{1}.
 \end{equation}
Moreover,  via the sphere coordinates, one can obtain $\|w\|_{1}$ as follows:
  \begin{equation}\label{G-fu}
  	\|w\|_{1}=\left\{\begin{aligned}
  		&\frac {(a-1/m)^n}{a^n}+n\left(\ln\frac {b}{a}+B(m)\right) \quad &\textrm{if }n\ge 2,\\
  		&\frac {a-1/m}{a}+\ln{\frac {b}{a}} \quad &\textrm{if }n=1,
  	\end{aligned}
  	\right.
  \end{equation}
  where
  \begin{gather*}
  B(m)=\sum \limits_{i=0}^{n-2}\left(C_{n-1}^{i} \frac {(-1/m)^{n-i-1}}{(i+1-n)b^{n-i-1}}-C_{n-1}^{i} \frac {(-1/m)^{n-i-1}}{(i+1-n)a^{n-i-1}}\right).
  \end{gather*}
 Let $m\rightarrow \infty$ for (\ref{G-fu}). We can  deduce from  (\ref{G-4-Mf}) that
     \begin{equation}\label{G-fu'}
   	\l\|M_a^b\r\|_{L^1}\ge
   		1+n\ln\frac{b}{a}.
   \end{equation}
Combining (\ref{G-dan-1}) with (\ref{G-fu'})  yields that
  \begin{equation*}
  \l\|M_a^b\r\|_{L^1}=1+n\ln\frac{b}{a}.
  \end{equation*}
  This  finishes the proof  Theorem \ref{main-2}.
\end{proof}

\section{The $L^1$-norm of the general strong truncated maximal operator $\tilde M_{\boldsymbol{a}}^{\boldsymbol{b}}$}\label{sec4}
As mentioned in Section \ref{sec1}, the proof of Theorem \ref{main-4} contains the proof of Theorem \ref{main-3}. In addition, from Lemma \ref{L2.1}, we have already known that the general strong truncated Hardy-Littlewood maximal operator $\tilde M_{\boldsymbol{a}}^{\boldsymbol{b}}$ is bounded on $L^1(\bR^n)$.  Therefore, in this section, we will focus on calculating the $L^1$-norm of  $\tilde M_{\boldsymbol{a}}^{\boldsymbol{b}}$, and readers can find that the first part of our proof is exactly the proof of Theorem \ref{main-3}.
\begin{proof}[Proof of Theorem \ref{main-4}]
	Similarly, we can restrict our discussion to $C_c(\bR^n)$. Now, let $\varphi\in C_c(\bR^n)$ satisfy $\|\varphi\|_{1}=1$. For any $\varepsilon>0$, following the same approach as in the proof of Theorem \ref{main-2}, there exists a simple function
	\begin{gather*}
	g=\left(2^{nm}\sum \limits_{i=1}^N\alpha_i\right)\chi_{O_{11}},
	\end{gather*}
	such that $\|g\|_1<1+\varepsilon$ and $\|\tilde M_{\boldsymbol{a}}^{\boldsymbol{b}}\varphi\|_{1}\le \|\tilde M_{\boldsymbol{a}}^{\boldsymbol{b}}g\|_{1}$, where $m,\, N \in \bN$, $\alpha_i>0$ with $i=1,2,\dots,N$, and $O_{11}$ is a cube in $\bR^n$ satisfying the following properties.
	\vspace{-0.3em}
\begin{enumerate}
\setlength{\itemsep}{1.5pt}
\setlength{\parsep}{0pt}
\setlength{\parskip}{0pt}
\setlength{\topsep}{1ex}
\item [(\it{a}\rm{)}] $O_{11}$ is centered at the origin of $\mathbb{R}^n$;
\item [(\it{b}\rm{)}] Via choosing $m$ sufficiently large, the side length of $Q_{11}$ can be less than $\min_{i}(\frac {a_i}{2\sqrt{n_i}})$.
\end{enumerate}

\vspace{-0.3em}
	If we denote by $2s$ the side length of $Q_{11}$, it follows that $s$ can be small enough with the assumption that $m$ is sufficiently large.
	
	\vspace{0.5em}
	{\it Part 1.} Let us first tackle the case $k=n$, which coincides with Theorem \ref{main-3}. For the sake of clarity in writing, we merely give the proof with the case $n = 2$, and the same is true for $n \neq 2$.
	
	For such a function $g$, we  can estimate $\tilde M_{\boldsymbol{a}}^{\boldsymbol{b}}g(x)$ via the scales of $x$, and the basic idea is dividing the whole space $\bR^2$ into 5 parts by two rectangles, that is,
	\begin{gather*}
	\prod_{i=1}^2[-a_i-s,a_i+s]
	\end{gather*}
	and
	\begin{gather*}
	\prod_{i=1}^2[-b_i-s,b_i+s].
	\end{gather*}
	Using the definition of  $\tilde M_{\boldsymbol{a}}^{\boldsymbol{b}}g(x)$, one can directly compute the following results.

	For $x\in \prod_{i=1}^2[-a_i-s,a_i+s]$, there simply holds
	\begin{gather*}
	\tilde M_{\boldsymbol{a}}^{\boldsymbol{b}}g(x)\le\frac{\|g\|_{1}}{{2^2a_1a_2}}.
	\end{gather*}
	For $x\in ([-b_1-s,b_1+s]\backslash [-a_1-s,a_1+s])\times[-a_2-s,a_2+s]$, we can conclude by a crucial calculation that 
	\begin{gather*}
	\tilde M_{\boldsymbol{a}}^{\boldsymbol{b}}g(x)\le \frac {\|g\|_{1}}{2^2a_2(|x_1|-s)}.
	\end{gather*}
	For $x\in [-a_1-s,a_1+s]\times ([-b_2-s,b_2+s]\backslash [-a_2-s,a_2+s])$, there similarly holds
	\begin{gather*}
	\tilde M_{\boldsymbol{a}}^{\boldsymbol{b}}g(x)\le \frac {\|g\|_{1}}{2^2a_1(|x_2|-s)}.
	\end{gather*}
	For $x\in \prod_{i=1}^2([-b_i-s,b_i+s]\backslash [-a_i-s,a_i+s])$, it can be analogically deduced that 
	\begin{gather*}
	\tilde M_{\boldsymbol{a}}^{\boldsymbol{b}}g(x)\le \frac {\|g\|_{1}}{2^2(|x_1|-s)(|x_2|-s)}.
	\end{gather*}
	At last, for $x\in (\prod_{i=1}^2[-b_i-s,b_i+s])^{c}$, we immediately have
	\begin{gather*}
	\tilde M_{\boldsymbol{a}}^{\boldsymbol{b}}g(x)=0.
	\end{gather*}
	
	Similarly, for the $n$ dimensional case, we can also split the whole space $\bR^n$ into $2^n+1$ pieces by two rectangles, namely,
	\begin{gather*}
	\prod_{i=1}^n[-a_i-s,a_i+s]
	\end{gather*}
	 and
	 \begin{gather*}
	 \prod_{i=1}^n[-b_i-s,b_i+s].
	 \end{gather*}
	
	Following these estimates, we take
	\begin{equation*}
		u(x)=\left\{\begin{aligned}
			&\frac{\|g\|_{1}}{{2^2a_1a_2}}\, & &\textrm{if }x\in \prod\limits _{i=1}^2[-a_i-s,a_i+s],\\
			&\frac {\|g\|_{1}}{2^2a_2(|x_1|-s)}\,&&\textrm{if }x\in \left([-b_1-s,b_1+s]\backslash [-a_1-s,a_1+s]\right)\times [-a_2-s,a_2+s],\\
			&\frac {\|g\|_{1}}{2^2a_1(|x_2|-s)}\,&&\textrm{if }x\in[-a_1-s,a_1+s]\times ([-b_2-s,b_2+s]\backslash [-a_2-s,a_2+s]),\\
			&\frac {\|g\|_{1}}{2^2(|x_1|-s)(|x_2|-s)}\,&&\textrm{if }x\in\prod\limits _{i=1}^2([-b_i-s,b_i+s]\backslash [-a_i-s,a_i+s]),\\
			&0\, & &\textrm{if }x\in \left(\prod_{i=1}^2[-b_i-s,b_i+s]\right)^c.
		\end{aligned}
		\right.
	\end{equation*}
	The $L^1$-norm of $u$ can be estimated as follows:
	\begin{equation}
	\label{new1}
	\begin{split}
			&\|u\|_{1}=\|g\|_{1}\left(\prod\limits _{i=1}^2\frac {a_i+s}{a_i}+\frac {a_2+s}{a_2}\ln\frac {b_1}{a_1}+\frac {a_1+s}{a_1}\ln\frac {b_2}{a_2}+\prod\limits _{i=1}^2\ln\frac{b_i}{a_i}\right).
		\end{split}
	\end{equation}

	For each $\varphi\in C_c(\bR^2)$ with $\|\varphi\|_{1}=1$, we have
	\begin{equation}\label{G-4-les2}
		\l\|\tilde M_{\boldsymbol{a}}^{\boldsymbol{b}}\varphi\r\|_{1}\le \l\|\tilde M_{\boldsymbol{a}}^{\boldsymbol{b}}g\r\|_{1}\le \|u\|_{1}
	\end{equation}
	from the choice of $g$ and $u$.
	Let $s\rightarrow 0$ in \eqref{new1} and take (\ref{G-4-les2}) into consideration, we get
	\begin{align*}
		\l\|\tilde M_{\boldsymbol{a}}^{\boldsymbol{b}}\varphi\r\|_{1}&\le (1+\varepsilon)\left(1+\sum_{i=1}^{2} \ln\frac {b_i}{a_i}+\prod\limits _{i=1}^2\ln\frac{b_i}{a_i}\right)=(1+\varepsilon)\prod\limits_{i=1}^{2}\left(1+\ln\frac {b_i}{a_i}\right).
	\end{align*}
	Since $\varepsilon>0$ is arbitrary, we have
	\begin{equation*}
		\l\|\tilde M_{\boldsymbol{a}}^{\boldsymbol{b}}\varphi\r\|_{1}\le\prod\limits_{i=1}^{2}\left(1+\ln\frac {b_i}{a_i}\right).
	\end{equation*}
	So, it naturally follows that
	\begin{equation}
	\label{new4}
		\l\|\tilde M_{\boldsymbol{a}}^{\boldsymbol{b}}\r\|_{L^1}\le \prod\limits_{i=1}^{2}\left(1+\ln\frac {b_i}{a_i}\right).
	\end{equation}

	Conversely, let
	\begin{gather*}
	f_m=\frac {\chi_{B(0,1/m)}}{V_2m^{-2}},
	\end{gather*}
	for $m=1,2,3,\dots$. Obviously,  $\|f_{m}\|_{1}=1$.
	
	For a sufficiently large $m$, take
	\begin{gather*}
	\begin{footnotesize}
	w(x)=\left\{\begin{aligned}
			&\frac{1}{{2^2a_1a_2}}&&\prod\limits _{i=1}^2[-a_i+1/m,a_i-1/m],\\
			&\frac {1}{2^2a_2(|x_1|+1/m)}&&([-b_1+1/m,b_1-1/m]\backslash [-a_1+1/m,a_1-1/m])\times [-a_2+1/m,a_2-1/m],\\
			&\frac {1}{2^2a_1(|x_2|+1/m)}&& [-a_1+1/m,a_1-1/m]\times ([-b_2+1/m,b_2-1/m]\backslash [-a_2+1/m,a_2-1/m]),\\
			&\frac {1}{2^2(|x_1|+1/m)(|x_2|+1/m)}&&\prod\limits _{i=1}^2([-b_i+1/m,b_i-1/m]\backslash [-a_i+1/m,a_i-1/m]),\\
			&0& &\left(\prod\limits _{i=1}^2[-b_i+1/m,b_i-1/m]\right)^c.
		\end{aligned}
		\right.
		\end{footnotesize}
	\end{gather*}
	
A direct computation yields that
	\begin{equation}\label{G-4-Mf2}
		\l\|\tilde M_{\boldsymbol{a}}^{\boldsymbol{b}}\r\|_{L^1}\geq\l\|\tilde M_{\boldsymbol{a}}^{\boldsymbol{b}}f_m\r\|_{1}\ge \|w\|_{1}.
	\end{equation}
The $L^1(\bR^2)$ norm of $w$ can be estimated as follows:
	\begin{equation}
	\label{new3}
	\begin{split}
			&||w||_{1}=\prod\limits _{i=1}^2\frac {a_i-1/m}{a_i}+\frac {a_2-1/m}{a_2}\ln\frac {b_1}{a_1}+\frac {a_1-1/m}{a_1}\ln\frac {b_2}{a_2}+\prod\limits _{i=1}^2\ln\frac{b_i}{a_i}.
		\end{split}
	\end{equation}
	Let $m\rightarrow \infty$ in \eqref{new3}
	and take account of (\ref{G-4-Mf2}), we have
	\begin{equation}\label{G-fu2'}
		\l\|\tilde M_{\boldsymbol{a}}^{\boldsymbol{b}}\r\|_{L^1}\ge\prod\limits_{i=1}^{2}\left(1+\ln\frac {b_i}{a_i}\right).
	\end{equation}
	Combining \eqref{new4} with (\ref{G-fu2'}), it yields
	\begin{equation*}
		\l\|\tilde M_{\boldsymbol{a}}^{\boldsymbol{b}}\r\|_{L^1}= \prod\limits_{i=1}^{2}\left(1+\ln\frac {b_i}{a_i}\right),
	\end{equation*}
which finishes the proof of part 1.

\vspace{0.5em}
{\it Part 2.}  Now we divert our attention to the case $k< n$. Without loss of generality, we assume that $n_i\geq 2$ for each $i=1,2,\dots,k$. Again, for simplicity and clarity, we only consider the case $k=2$, and the same is also true for $k\neq 2$.

We further divide $\bR^{n_1}\times\bR^{n_2}$ into 5 areas by using two general rectangles, that is,
\begin{gather*}
\prod_{i=1}^{2}B(0, a_i+\sqrt{n_i}s)
\end{gather*}
and
\begin{gather*}
\prod_{i=1}^{2}B(0,b_i+\sqrt{n_i}s),
\end{gather*}
where $B(0, a_i+\sqrt{n_i}s),\,B(0,b_i+\sqrt{n_i}s)\subset\bR^{n_i}$ with $i=1,2$. Analogously, for the general case $k\neq 2$, $\bR^{n_1}\times\bR^{n_2}\times \cdots \bR^{n_k}$ will be partitioned into $2^n+1$ parts by
two general rectangles, namely,
\begin{gather*}
\prod_{i=1}^{k}B(0, a_i+\sqrt{n_i}s)
\end{gather*}
and
\begin{gather*}
\prod_{i=1}^{k}B(0,b_i+\sqrt{n_i}s),
\end{gather*}
where $B(0, a_i+\sqrt{n_i}s),\,B(0,b_i+\sqrt{n_i}s)\subset\bR^{n_i}$ with $i=1,2,\dots,k$.

	Set
	\begin{equation*}
		u(x)=\left\{\begin{aligned}
			&\frac{\|g\|_{1}}{{V_{n_1}V_{n_2}a_1^{n_1}a_2^{n_2}}}&&\textrm{if }x\in  \prod_{i=1}^2B(0, a_i+\sqrt{n_i}s),\\
			&\frac {\|g\|_{1}}{V_{n_1}V_{n_2}a_2^{n_2}(|x_1|-\sqrt{n_1}s)^{n_1}}&&\textrm{if }x\in \left(B(0,b_1+\sqrt{n_1}s)\backslash B(0, a_1+\sqrt{n_1}s)\right)\times B(0, a_2+\sqrt{n_2}s),\\
			&\frac {\|g\|_{1}}{V_{n_1}V_{n_2}a_1^{n_1}(|x_2|-\sqrt{n_2}s)^{n_2}}&&\textrm{if }x\in B(0, a_1+\sqrt{n_1}s)\times (B(0,b_2+\sqrt{n_2}s)\backslash B(0, a_2+\sqrt{n_2}s)),\\
			&\frac {\|g\|_{1}}{\prod_{i=1}^2V_{n_i}(|x_i|-\sqrt{n_i}s)^{n_i}}&&\textrm{if }x\in\prod\limits _{i=1}^2(B(0,b_i+\sqrt{n_i}s)\backslash B(0, a_i+\sqrt{n_i}s)),\\
			&0& &\textrm{if }x\in \left(\prod_{i=1}^2B(0,b_i+\sqrt{n_i}s)\right)^c.
		\end{aligned}
		\right.
	\end{equation*}
It follows from a direct calculation that
\begin{gather}
	\label{Ge1.2}
	\l\|\tilde M_{\boldsymbol{a}}^{\boldsymbol{b}}\varphi\r\|_{1}\le \l\|\tilde M_{\boldsymbol{a}}^{\boldsymbol{b}}g\r\|_{1}\leq \|u\|_1,
\end{gather}
and
\begin{equation}
\begin{split}
\label{Ge1.1}
 \|u\|_{1}=\|g\|_1&\Bigg(\prod_{i=1}^{2}\frac {(a_i+\sqrt{n_i}s)^{n_i}}{a_i^{n_i}}+\frac {(a_2+\sqrt{n_2}s)^{n_2}}{a_2^{n_2}}n_1\left(\ln\frac {b_1}{a_1}+A_1(s)\right)\\
 &+\frac {(a_1+\sqrt{n_1}s)^{n_1}}{a_1^{n_1}}n_2\left(\ln\frac {b_2}{a_2}+A_2(s)\right)+\prod_{i=1}^2n_i\left(\ln\frac {b_i}{a_i}+A_i(s)\right)\Bigg),
 \end{split}
\end{equation}
where
\begin{gather*}
A_i(s)=\sum \limits_{j=0}^{n_i-2}\left(C_{n_i-1}^{j} \frac {(\sqrt{n_i}s)^{n_i-j-1}}{(j+1-n_i)b_i^{n_i-j-1}}-C_{n_i-1}^{j} \frac {(\sqrt{n_i}s)^{n_i-j-1}}{(j+1-n_i)a_i^{n_i-j-1}}\right)
\end{gather*}
for $i=1,2$.

Let $s\to 0$ in \eqref{Ge1.1} and take account of \eqref{Ge1.2}, we obtain that
\begin{align*}
	\l\|\tilde M_{\boldsymbol{a}}^{\boldsymbol{b}}\varphi\r\|_{1}&\leq (1+\varepsilon)\left(1+\sum_{i=1}^{2} n_i\ln\frac {b_i}{a_i}+\prod\limits _{i=1}^2n_i\ln\frac{b_i}{a_i}\right)\\
	&=(1+\varepsilon)\prod\limits_{i=1}^{2}\left(1+n_i\ln\frac {b_i}{a_i}\right).
\end{align*}
Then the arbitrariness of $\varepsilon>0$ and $\varphi\in C_c(\bR^n)$ with $\|\varphi\|_1=1$ give that
\begin{gather}
\label{Ge1.3}
\l\|\tilde M_{\boldsymbol{a}}^{\boldsymbol{b}}\r\|_{L^1}\leq \prod\limits_{i=1}^{2}\left(1+n_i\ln\frac {b_i}{a_i}\right).
\end{gather}

Conversely, inspired by the shape of function $u$, we construct again a sequence of functions to deduce the inverse inequality.

Define
\begin{gather*}
f_m:=\frac {\chi_{B(0,1/m)}}{V_nm^{-n}},
\end{gather*}
for $m=1,2,3,\dots$. Obviously,  $\|f_{m}\|_{1}=1$.

For $m$ large enough, take
	\begin{gather*}
	w(x)=\left\{\begin{aligned}
			&\frac{1}{{V_{n_1}V_{n_2}a_1^{n_1}a_2^{n_2}}}&&\textrm{if }x\in  \prod_{i=1}^2B(0, a_i-1/m),\\
			&\frac {1}{V_{n_1}V_{n_2}a_2^{n_2}(|x_1|+1/m)^{n_1}}&&\textrm{if }x\in \left(B(0,b_1-1/m)\backslash B(0, a_1-1/m)\right)\times B(0, a_2-1/m),\\
			&\frac {1}{V_{n_1}V_{n_2}a_1^{n_1}(|x_2|+1/m)^{n_2}}&&\textrm{if }x\in B(0, a_1-1/m)\times (B(0,b_2-1/m)\backslash B(0, a_2-1/m)),\\
			&\frac {1}{\prod_{i=1}^2V_{n_i}(|x_i|+1/m)^{n_i}} &&\textrm{if }x\in\prod\limits _{i=1}^2(B(0,b_i-1/m)\backslash B(0, a_i-1/m)),\\
			&0 & &\textrm{if }x\in \left(\prod_{i=1}^2B(0,b_i-1/m)\right)^c.
		\end{aligned}
		\right.
	\end{gather*}
A direct calculation yields that
\begin{gather}
\label{Ge1.4}
\l\|\tilde M_{\boldsymbol{a}}^{\boldsymbol{b}}\r\|_{L^1}\geq \l\|\tilde M_{\boldsymbol{a}}^{\boldsymbol{b}}f_m\r\|_{1}\geq \|w\|_1,
\end{gather} 	
and
\begin{equation}
\label{Ge1.5}
\begin{split}
 \|w\|_{1}=&\prod_{i=1}^{2}\frac {(a_i-1/m)^{n_i}}{a_i^{n_i}}+\frac {(a_2-1/m)^{n_2}}{a_2^{n_2}}n_1\left(\ln\frac {b_1}{a_1}+B_1(m)\right)\\
 &+\frac {(a_1-1/m)^{n_1}}{a_1^{n_1}}n_2\left(\ln\frac {b_2}{a_2}+B_2(m)\right)+\prod_{i=1}^2n_i\left(\ln\frac {b_i}{a_i}+B_i(m)\right),
 \end{split}
\end{equation}
where
\begin{gather*}
B_i(m)=\sum \limits_{j=0}^{n_i-2}\left(C_{n_i-1}^{j} \frac {(-1/m)^{n_i-j-1}}{(j+1-n_i)b_i^{n_i-j-1}}-C_{n_i-1}^{j} \frac {(-1/m)^{n_i-j-1}}{(j+1-n_i)a_i^{n_i-j-1}}\right)
\end{gather*}
for $i=1,2$.

Consider \eqref{Ge1.4}, and let $m\to\infty$ in \eqref{Ge1.5}. We then deduce that
\begin{gather}
\label{newaddl2}
\l\|\tilde M_{\boldsymbol{a}}^{\boldsymbol{b}}\r\|_{L^1}\geq \prod\limits_{i=1}^{2}\left(1+n_i\ln\frac {b_i}{a_i}\right).
\end{gather}
Finally, \eqref{Ge1.3} and \eqref{newaddl2} give that
\begin{gather*}
\l\|\tilde M_{\boldsymbol{a}}^{\boldsymbol{b}}\r\|_{L^1}= \prod\limits_{i=1}^{2}\left(1+n_i\ln\frac {b_i}{a_i}\right),
\end{gather*}
which completes the proof of Part 2, and hence, the whole proof is finished.	
\end{proof}

\section*{Acknowledgements}
This work is supported by the National Natural Science Foundation of China (Grant No.s\,11871452 and 12071052), the Natural Science Foundation of
Henan (Grant No.\,202300410338), and the Nanhu Scholar Program for Young Scholars of XYNU.

\end{document}